\documentclass[11pt,reqno]{amsart}
\usepackage[english]{babel}
\usepackage{amssymb,amsmath}
\usepackage[mathscr]{eucal}
\usepackage{url,ushort,bbm}
\usepackage{enumerate}
\usepackage{hyperref}
\usepackage{esvect}

\usepackage{color,comment}
\usepackage[initials]{amsrefs}

\theoremstyle{plain}
\newtheorem{prop}{Proposition}[section]
\newtheorem{thm}[prop]{Theorem}
\newtheorem{cor}[prop]{Corollary}
\newtheorem{lemma}[prop]{Lemma}
\newtheorem{con}[prop]{Conjecture}

\theoremstyle{definition}
\newtheorem{dfn}[prop]{Definition}
\newtheorem{rem}[prop]{Remark}

\newcommand{\om}{\omega}
\newcommand{\ep}{\varepsilon}
\renewcommand{\S}{\mathcal S}

\newcommand{\N}{\mathbb N}

\newcommand{\R}{\mathbb R}
\newcommand{\C}{\mathbb C}

\renewcommand{\i}{\mathbbm i}

\newcommand{\kk}{\mathbbm k}
 
\newcommand{\ax}{\langle {\ushort {X}}\rangle}

\newcommand{\csim}{\stackrel{\mathrm{cyc}}{\thicksim}}

\DeclareMathOperator{\tr}{tr}
\DeclareMathOperator{\Tr}{Tr}

\newcommand{\mato}{\left(\begin{smallmatrix}}
\newcommand{\matc}{\end{smallmatrix}\right)}

\newcommand{\m}{\mathcal M}
\newcommand{\n}{\mathcal N}

\newcommand{\sym}{\mathrm{sym}}
\newcommand{\antisym}{\mathrm{antisym}}
\newcommand\op{{\mathrm{op}}}
\newcommand\id{{\operatorname{id}}}

\begin{document}
\title[Connes' embedding conjecture]{Correction of a proof in "Connes' embedding conjecture and sums of hermitian squares"}

\author[Burgdorf]{Sabine Burgdorf$^\dag$}
\address{S.B., EPFL, SB-MATHGEOM-EGG, Station 8, 1015 Lausanne, Switzerland}
\email{sabine.burgdorf@epfl.ch}
\thanks{\footnotesize ${}^{\dag}$Research partially supported by ERC}

\author[Dykema]{Ken Dykema$^{*}$}
\address{K.D., Department of Mathematics, Texas A\&M University,
College Station, TX 77843-3368, USA}
\email{kdykema@math.tamu.edu}
\thanks{\footnotesize ${}^{*}$Research supported in part by NSF grant DMS-1202660.}

\author[Klep]{Igor Klep}
\address{I.K., Department of Mathematics, The University of Auckland, 
Private Bag 92019, Auckland 1142, New Zealand}
\email{igor.klep@auckland.ac.nz }

\author[Schweighofer]{Markus Schweighofer}
\address{M.S., FB Mathematik und Statistik,
Universitat Konstanz, D-78457 Konstanz,  Germany}
\email{markus.schweighofer@uni-konstanz.de}



\date{March 13, 2013}

\begin{abstract}
We show that Connes' embedding conjecture (CEC) is equivalent to a real version of the same (RCEC).
Moreover, we show that
RCEC is equivalent to a real, purely algebraic statement concerning trace positive polynomials.
This purely
algebraic reformulation of CEC had previously been given in both a real and a complex version in a paper of the last two
authors.
The second author discovered a gap in this earlier proof of the equivalence of CEC to the real algebraic reformulation (the proof of the complex algebraic reformulation being correct).
In this note, we show that this gap can be filled with help of the theory of real von Neumann algebras.
\end{abstract}

\maketitle

\vspace{-2pt}
\section{Introduction and erratum for \cite{ksconnes}}

Alain Connes stated in 1976 the following conjecture \cite{connes}*{Section V}.

\begin{con}[CEC]\label{concon}
If $\om$ is a free ultrafilter on $\N$ and $\mathcal F$ is a 
II$_1$-factor with separable predual, then 
$\mathcal F$ can be embedded into an ultrapower $\mathcal R^\om$ of the hyperfinite II$_1$-factor $\mathcal R$.
\end{con}

The last two authors gave a purely algebraic statement which
is equivalent to Conjecture \ref{concon}, {\em cf.} statements (i) and (ii) in \cite{ksconnes}*{Thm.\ 3.18}.
Before stating this
we recall some notation used in~\cite{ksconnes}.
For $\kk\in\{\R,\C\}$, $\kk\ax$ denotes the polynomial ring in $n$ non-commuting self-adjoint variables $\ushort X=(X_1,\dots,X_n)$ over $\kk$, which is equipped with the natural involution $f\mapsto f^*$,
i.e.\ it is the natural involution on $\kk$, fixes each $X_i$ and reverses the order of words.
Then $M_\kk$ denotes the quadratic module in $\kk\ax$ generated by $\{1-X_i^2\mid i=1,\dots,n\}$.
Two polynomials $f,g\in\kk\ax$ are said to be cyclically equivalent ($f\csim g$)
if $f-g$ is a sum of commutators in $\kk\ax$. 

\begin{thm}[Klep, Schweighofer]\label{thm:algcon}
The following statements are equivalent:
\begin{enumerate}[{\rm(a)}]\itemsep0pt
\item 
{\rm CEC} is true.
\item If $f\in\C\ax$ and if $\tr(f(\ushort A))\ge0$ for all tuples $\ushort A$ of
self-adjoint contractions in $\C^{s\times s}$ for all $s\in\N$,
then for every $\ep\in\R_{>0}$, $f+\ep$ is cyclically equivalent to an element of $M_\C$.
\end{enumerate}
\end{thm}

In the same paper the authors gave the following real version of the purely algebraic reformulation of Connes' embedding conjecture.

\begin{thm}\label{thm:Ralgcon}
The following statements are equivalent:
\begin{enumerate}[{\rm(a)}]\itemsep0pt
\item 
{\rm CEC} is true.
\item If $f=f^*\in\R\ax$ and if $\tr(f(\ushort A))\ge0$ for all tuples $\ushort A$ of
symmetric contractions in $\R^{s\times s}$ for all $s\in\N$,
then for every $\ep\in\R_{>0}$, $f+\ep$ is cyclically equivalent to an element of $M_\R$.
\end{enumerate}
\end{thm}

However, their proof of the implication (b)$\implies$(a) in
Theorem~\ref{thm:Ralgcon} is not correct.
The incorrect part of that argument is Proposition~2.3 of~\cite{ksconnes};
in fact, the polynomial $f=\i(X_1X_2X_3-X_3X_2X_1)$
provides a counter-example to the statement of that proposition.
In this note, we present a proof of Theorem~\ref{thm:Ralgcon} that uses real von Neumann algebras
as well as techniques 
and results from~\cite{ksconnes}.

An inspection of the proof of \cite[Proposition 2.3]{ksconnes}
shows that a weaker version of it remains true, namely the version 
where $\R$ is replaced by $\C$ in its formulation. This weaker version is
enough for the first application of \cite[Proposition 2.3]{ksconnes}, namely, in the proof of
\cite[Theorem 3.12]{ksconnes}.
It is the second application, namely, in the proof of \cite[Theorem 3.18]{ksconnes}, that is
illegitimate and will be circumvented by this note.
In particular, it will follow that
the statements of all results in \cite{ksconnes} are correct with the exception of \cite[Proposition 2.3]{ksconnes}.

We note that Narutaka Ozawa~\cite{Oz} provides a proof that uses similar (though not precisely identical) methods
to the one presented here;
our proofs were begun independently and, initially, completed independently;
however, after release of a first version of this paper, Ozawa noticed a problem with our proof and
provided a result (Proposition~\ref{prop:Taka} below) that fixed it, which he kindly allows us
to print here.

\section{Real von Neumann algebras}
\label{sec:RvN}

Real von Neumann algebras
were first systematically studied in the 1960s 
by E.\ St\o{}rmer (see~\cite{St68}, \cite{St67}).
They are closely related to von Neumann algebras with  involutory $*$-anti\-auto\-morph\-isms.
Through this relation much of the structure theory of von Neumann algebras
(e.g.\ the type classification and the integral decomposition into factors)  can be transferred to the real case.
See, for example, \cite{aru} (or~\cite{Li}, where the definition is slightly different but easily seen to be equivalent).

\begin{dfn}\label{efn:rvN}
A real von Neumann algebra $\m_r$ is a unital, weakly closed, real, self-adjoint subalgebra of the (real) algebra
bounded linear operators
on a complex Hilbert space, with the property $\m_r\cap \i\m_r=\{0\}$.
\end{dfn}

\begin{rem}\label{rem:complexification}
A basic fact of real von Neumann algebras (see, e.g., the introduction of \cite{aru}, or references cited therein)
is that they correspond 
to (complex) von Neumann algebras with involutory $*$-antiautomorphism.
An involutory 
$*$-antiautomorphism on a von Neumann algebra $\m$ is a complex linear map $\alpha:\m\to\m$ satisfying
$\alpha(x^*)=\alpha(x)^*$,
$\alpha(xy)=\alpha(y)\alpha(x)$ and $\alpha^2(x)=x$ for all $x,y\in\m$. 
This correspondence works as follows.
\begin{enumerate}[(i)]
\item Let $\m_r$ be a real von Neumann algebra. Then the (complex) von Neumann algebra $\mathcal U(\m_r)$ 
generated by $\m_r$ is equal to the complexification of $\m_r$ \cite{St68}*{Thm.\ 2.4}, i.e. $$\mathcal U(\m_r)=\m_r''=\m_r+\i\m_r.$$ 
Moreover, the involution $*$ on $\m_r$ generates a natural involutory $*$-antiautomorphism $\alpha$ on 
$\mathcal U(\m_r)$ by $$\alpha(x+\i y)=x^*+\i y^*.$$
\item Let $\m$ be a von Neumann algebra with involutory $*$-antiautomorphism $\alpha$, the $*$-subalgebra
 $$\m_\alpha=\{x\in\m\mid \alpha(x)=x^*\}$$ 
is then a real von Neumann algebra. In fact, let $x=\i y\in\m_\alpha\cap\i\m_\alpha$ with $x,y\in\m_\alpha$, then 
$x^*=\alpha(x)=i\alpha(y)=(-y)^*=-x^*$ and thus $x^*=0$, which implies $\m_\alpha\cap\i \m_\alpha=\{0\}$. 
Since $*$ and $\alpha$ are continuous in the weak topology, $\m_\alpha$ is weakly closed.
\end{enumerate}
One then easily sees that $\mathcal U(\m_\alpha)=\m$ and $\mathcal U(\m_r)_{\alpha}=\m_r$. 
\end{rem}

\begin{dfn} Let $\m_r$ be a real von Neumann algebra. 
\begin{enumerate}[(i)]\itemsep0pt
\item $\m_r$ is called a real factor if its center $Z(\m_r)$ consists of only the real scalar operators. 
\item $\m_r$ is said to be hyperfinite if there exists an increasing sequence of finite dimensional
real von Neumann subalgebras of $\m_r$ such that its union is weakly dense in $\m_r$.
\item $\m_r$  of type I$_n$,  I$_\infty$, II$_1$, etc. if its complexification $\mathcal U(\m_r)$ is of the corresponding type.
\end{enumerate}
\end{dfn}

We immediately see that a real von Neumann algebra $\m_r$ is a factor if and only if $\mathcal U(\m_r)$ is a factor.

If $\m_r$ is a real hyperfinite factor, then $\mathcal U(\m_r)$ is a hyperfinite factor \cite{aru}*{Prop.\ 2.5.10}. 
But the converse implication does not hold in general, see e.g. \cite{aru}*{Ch.\ 2.5}. The situation is different 
for the hyperfinite II$_1$-factors. There is (up to isomorphism) a unique real hyperfinite II$_1$-factor 
\cite{St80}*{Thm.\ 2.1} (see also~\cite{G83}).

\begin{thm}[St\o rmer] \label{realhyper}
Let $\m$ be a type II$_1$-factor and $\m_r$ a real factor such that $\m=\m_r+\i \m_r$. Then the following conditions are equivalent.
\begin{enumerate}[\rm(i)]\itemsep0pt
\item $\m_r$ is hyperfinite.
\item $\m_r$ is the weak closure of the union of an increasing sequence $\{R_n\}$ of real unital
factors such that $R_n$ is isomorphic to $\R^{2^n\times 2^n}$.
\item $\m_r$ is countably generated, and given $x_1,\dots,x_n\in \m_r$  and $\ep>0$ there exist a
finite dimensional real von Neumann subalgebra $\n_r$ of $\m_r$ and $y_1,\dots,y_n\in \n_r$ such
that $\|y_k-x_k \|_2<\ep$ for all $k=1,\dots,n$. 
\end{enumerate}
\end{thm}
From now on we will denote the unique real hyperfinite II$_1$-factor by $\mathcal R_r$. 

\begin{rem}\label{rem:power}
The correspondence between $\mathcal R$ and $\mathcal R_r$ in the hyperfinite case of
Remark \ref{rem:complexification} is given by the   
involutory $*$-antiautomorphism on $\mathcal R$ which is induced by the matrix transpose.
To be more specific, with $\mathcal R$ the closure of the infinite tensor product of matrix algebras $\C^{2\times 2}$
and letting $t_2$ be the
matrix transpose on $\C^{2\times 2}$,
then  we define $\alpha$ to be the the involutory $*$-antiautomoprhism on $\mathcal R$
that when restricted to $\bigotimes_1^\infty \C^{2\times 2}$
is $\otimes_1^\infty t_2$;
see also \cite{St80}*{Cor.\ 2.10}.
Then $\mathcal R_\alpha=\mathcal R_r$ and hence 
$\mathcal R=\mathcal R_r+\i\mathcal R_r$. 
\end{rem}

The construction of the ultrapower of the real hyperfinite II$_1$-factor $\mathcal R_r$ with trace $\tau$ works as in the complex case, see e.g. \cite{St80}. 
Let $\om$ be a free ultrafilter on $\N$. Parallel to $\ell^\infty(\mathcal R)$ in the complex case consider the 
real $C^*$-algebra 
$\ell^\infty(\mathcal R_r)=\{(r_k)_{k\in\N}\in\mathcal R_r^\N \mid \sup_{k\in \N} \|r_k\|<\infty\}.$ 
Further let $J_\om=\{(r_k)_k\in\ell^\infty(\mathcal R_r)\mid  \lim_{k\to\om}\tau(r_k^*r_k)^{1/2}=0\}$. 
Then $J_\om$ is a closed maximal ideal in $\ell^\infty(\mathcal R_r)$. The quotient $C^*$-algebra 
$\mathcal R_r^\om:=\ell^\infty(\mathcal R_r)/J_\om$ is called the ultrapower of $\mathcal R_r^\om$
and is a finite real von Neumann algebra. 
Since $I_\om=J_\om+\i J_\om$, where $I_\om$ is the closed ideal in 
$\ell^\infty(\mathcal R)$ used for the construction of $\mathcal R^\om$, we have
$\mathcal R^\omega=(\mathcal R_r+\i\mathcal R_r)^\omega=\mathcal R_r^\omega+\i\mathcal R_r^\omega$. 

\medskip
A final topic in this section is related to generating sets of real von Neumann algebras.
For $\m_r$ a real von Neumann algebra, an element $x$ of $\m_r$ is said to be {\em symmetric} if $x^*=x$
and {\em antisymmetric} if $x^*=-x$.
Clearly, writing an arbitrary $x\in\m$ as
\begin{equation}\label{eq:symantisym}
x=\frac{x+x^*}2+\frac{x-x^*}2,
\end{equation}
every element of $\m_r$ is the sum of of a symmetric and an antisymmetric element.

\begin{lemma}\label{lem:symgen}
Let $\m_r$ be any real von Neumann algebra and consider the real von Neumann algebra
$M_2(\m_r)$ endowed with the usual adjoint operation:
\begin{equation}\label{eq:adj}
\left(\begin{matrix}a&b\\c&d\end{matrix}\right)^*=
\left(\begin{matrix}a^*&c^*\\b^*&d^*\end{matrix}\right).
\end{equation}
Then $M_2(\m_r)$ has a generating set consisting of symmetric elements.
\end{lemma}
\begin{proof}
Let $S\subseteq \m_r$ be a generating set for $\m_r$.
Using the trick~\eqref{eq:symantisym}, we may without loss of generality assume $S=S_\sym\cup S_\antisym$,
where $S_\sym$ consists of symmetric elements and $S_\antisym$ of antisymmetric elements.
Then
\[
\left\{
\left(\begin{matrix}1&0\\0&0\end{matrix}\right),\;
\left(\begin{matrix}0&1\\1&0\end{matrix}\right)
\right\}
\cup
\left\{
\left(\begin{matrix}x&0\\0&0\end{matrix}\right)\bigg| {}\; x\in\S_\sym
\right\}
\cup
\left\{
\left(\begin{matrix}0&x\\-x&0\end{matrix}\right)\bigg|{}\; x\in\S_\antisym
\right\}
\]
is a generating set for $M_2(\m_r)$ consisting of symmetric elements.
\end{proof}

We thank Narutaka Ozawa for providing the proof of the following proposition and for
allowing us to present it here.

\begin{prop}\label{prop:Taka}
Every real von Neumann algebra of type II$_1$
is isomorphic to $M_2(\m_r)$, for some real von Neumann algebra $\m_r$,
endowed with the usual adjoint operation~\eqref{eq:adj}.
\end{prop}
\begin{proof}
Let $\n$ be a (complex) von Neumann algebra of type II$_1$ endowed with
an involutory $*$-antiautomorphism $\alpha$ and let $\n_r=\{x\in\n\mid \alpha(x)=x^*\}$
be the corresponding real von Neumann algebra.
It will suffice to find a partial isometry $v\in\n_r$ such that $v^*v+vv^*=1$,
for this will imply that $\n_r$ has the desired $2\times 2$ matrix structure.
By a maximality argument, it will suffice to find a nonzero partial isometry $w\in\n_r$
so that $w^*w\perp ww^*$.
Let $E$ be the center-valued trace on $\n$.
Let $p\in\n$ be any nonzero projection such that $E(p)\le\frac13$ and let $q=p\vee\alpha(p)$.
Then $q\in\n_r$ and $E(q)\le\frac23$.
There is a nonzero element $x\in\n$ such that $x^*x\le q$ and $xx^*\le1-q$.
Multiplying by the imaginary number $\i$, if necessary, we may without loss of generality assume
$y=x+\alpha(x^*)$ is nonzero.
But $y\in\n_r$ and $y=(1-q)yq$.
In the polar decomposition $y=w|y|$ of $y$, we have that $w\in\n_r$ is a nonzero partial isometry,
$w^*w\le q$ and $ww^*\le(1-q)$.
\end{proof}

Combining the previous proposition and lemma, we get:
\begin{cor}\label{cor:II1symm}
Every real von Neumann algebra of type II$_1$ is generated by a set of symmetric elements.
\end{cor}


\section{The real version of Connes' embedding conjecture and proof of the real algebraic reformulation}
\label{sec:realcon}

Here is the real version of Connes' embedding conjecture:

\begin{con}[RCEC]\label{realconcon}
If $\om$ is a free ultrafilter on $\N$ and $\mathcal F_r$ is a real
II$_1$-factor with separable predual, then 
$\mathcal F_r$ can be embedded into an ultrapower $\mathcal R_r^\om$ of the real hyperfinite II$_1$-factor.
\end{con}

We will prove that RCEC is equivalent to CEC, in the course of proving Theorem~\ref{thm:Ralgcon}.
We will use two preliminary results.
The first of these is the following real analogue of~\cite{ksconnes}*{Prop.\ 3.17}.

\begin{prop}\label{realprop317}
For every real II$_1$-factor $\mathcal F_r$ with separable predual and faithful 
trace $\tau$, the following statements are equivalent:
\begin{enumerate}[\rm(i)]\itemsep0pt
\item\label{realallultra}
For every free ultrafilter $\om$ on $\N$, $\mathcal F_r$ is embeddable
in $\mathcal R_r^\omega$;
\item\label{realoneultra} There is an ultrafilter $\om$ on $\N$
such that $\mathcal F_r$ is embeddable in $\mathcal R_r^\om$; 
\item\label{realputinar} For each $n\in\N$ and $f=f^*\in\R\ax$, having positive trace on all symmetric contractions in $\R^{s\times s}, s\in\N$, 
implies that $\tau(f(\ushort A))\geq0$ for all symmetric contractions $\ushort A\in\mathcal F_r^n$;
\item\label{realapproximate}
For all $\ep\in\R_{>0}$, $n,k\in\N$ and symmetric contractions
$A_1,\dots,A_n\in\mathcal F_r$, there is an $s\in\N$ and symmetric contractions
$B_1,\dots,B_n\in\R^{s\times s}$ such that for all $w\in\ax_k$:
$$|\tau(w(A_1,\dots,A_n))-\Tr(w(B_1,\dots,B_n))|<\ep,$$
where $\Tr$ is the normalized trace on  $s\times s$ matrices.
\end{enumerate}
\end{prop}

\begin{proof}
The implication \eqref{realallultra}$\implies$\eqref{realoneultra} is obvious. The proof of 
\eqref{realoneultra}$\implies$\eqref{realputinar}$\implies$\eqref{realapproximate} works as in the complex case
(see the proof of \cite{ksconnes}*{Prop.\ 3.17}),
where for \eqref{realoneultra}$\implies$\eqref{realputinar} one uses the fact (see Theorem \ref{realhyper})
that $\mathcal R_r$ is generated by a union of an increasing 
sequence of  real matrix algebras. 

For the implication \eqref{realapproximate}$\implies$\eqref{realallultra}, 
by Corollary~\ref{cor:II1symm}, $\mathcal F_r$ has a generating set $A_1,A_2,\dots$ that is a sequence of symmetric contractions.
The rest of the proof then works as in the complex case, {\em cf.} also Theorem \ref{realhyper}(iii).
\end{proof}

For a von Neumann algebra $\m$, we let $\m^\op$ denote the opposite von Neumannn algebra of $\m$.
This is the algebra that is equal to
$\m$ as a set and with the same $*$-operation, but with multiplication operation $\cdot$
defined by $a\cdot b=ba$.
It is well known that $\m^\op$ is also a von Neumann algebra, (and it is very easy to see this in the case that
$\m$ has a normal faithful tracial state).
A $*$-antiautomorphism of $\m$ is precisely an isomorphism $\m\to\m^\op$ of von Neumann algebras.

The following proposition is true for arbitrary von Neumann algebras,
but since here we need it only for finite von Neumann algebras,
for convenience and ease of proof we state it only in this case.

\begin{prop}\label{prop:tens}
Let $\m$ be a von Neumann algebra with a normal, faithful, tracial state $\tau$.
Then the von Neumann algebra tensor product $\m\overline{\otimes}\m^\op$ has an involutory $*$-antiautomorphism.
\end{prop}
\begin{proof}
Let $\n=\m\overline{\otimes}\m^\op$
and let $\alpha:\m\otimes\m^\op\to\n$ be the linear map defined on the algebraic tensor product that satisfies
$\alpha(a\otimes b)=b\otimes a$.
Then $\alpha$ is $*$-preserving, and antimultiplicative.
Indeed, we have
\[
\alpha\left((a\otimes b)(c\otimes d)\right)
=\alpha(ac\otimes db)
=db\otimes ac
=(d\otimes c)(b\otimes a)
=\alpha(c\otimes d)\,\alpha(a\otimes b).
\]
Since $\alpha$ is trace-preserving, it extends to an isomorphism $\n\to\n^\op$ of von Neumann algebras,
i.e., a $*$-antiautomorphism of $\n$.
Moreover, it is clear that $\alpha^2=\id$.
\end{proof}

Now we are ready to prove Theorem~\ref{thm:Ralgcon}.
For convenience, the two conditions in that theorem are restated here along with a third, which we will show are all
equivalent.
\begin{thm}
The following statements are equivalent:
\begin{enumerate}[{\rm(a)}]\itemsep0pt
\item 
{\rm CEC} is true.
\item If $f=f^*\in\R\ax$ and if $\tr(f(\ushort A))\ge0$ for all tuples $\ushort A$ of
symmetric contractions in $\R^{s\times s}$ for all $s\in\N$,
then for every $\ep\in\R_{>0}$, $f+\ep$ is cyclically equivalent to an element of $M_\R$.
\item 
{\rm RCEC} is true.
\end{enumerate}
\end{thm}
\begin{proof}
As remarked after the statement of Theorem~\ref{thm:Ralgcon}, (a)$\implies$(b) was proved in~\cite{ksconnes}.

The implication (b)$\implies$(c) follows from Proposition~\ref{realprop317}.
Indeed, if $\mathcal F_r$ is a real II$_1$-factor with separable predual, then (b) implies that condition
(iii) of Proposition~\ref{realprop317} holds;
by~(i) of that proposition, it follows that $\mathcal F_r$ embeds in $\mathcal R^\om_r$.

For (c)$\implies$(a), we assume that {\rm RCEC} is true and we will show that every 
II$_1$-factor $\mathcal F$ with separable predual can be embedded into $\mathcal R^\om$.
Suppose first that $\mathcal F$ has an involutory $*$-antiautomorphism $\alpha$.
Then $\mathcal F$ can be written as 
$\mathcal F=\mathcal F_r+\i\mathcal F_r$, where $\mathcal F_r$ is the real II$_1$-factor inside $\mathcal F$ 
corresponding to $\alpha$ as in Remark \ref{rem:complexification}.
By {\rm RCEC} there exists an embedding $\iota$ of $\mathcal F_r$ into 
$\mathcal R_r^\om$. This implies by $\C$-linear extension of $\iota$ that $\mathcal F=\mathcal F_r+\i\mathcal F_r$ 
embeds into $\mathcal R_r^\omega+\i\mathcal R_r^\om=(\mathcal R_r+\i\mathcal R_r)^\omega=\mathcal R^\om$.
Now if $\mathcal F$ is any II$_1$-factor, by the above case and Proposition~\ref{prop:tens}, the
II$_1$-factor $\mathcal F\overline\otimes\mathcal F^\op$ embeds in $R^\om$, and from the identification of $\mathcal F$
with $\mathcal F\otimes1\subseteq\mathcal F\overline\otimes\mathcal F^\op$, we get that $\mathcal F$ embeds in $R^\om$.
\end{proof}


\begin{bibdiv}
\begin{biblist}


\bib{aru}{book}{
   author={Ayupov, Shavkat},
   author={Rakhimov, Abdugafur},
   author={Usmanov, Shukhrat},
   title={Jordan, real and Lie structures in operator algebras},
   series={Mathematics and its Applications},
   volume={418},
   publisher={Kluwer Academic Publishers Group},
   place={Dordrecht},
   date={1997},
}

\bib{connes}{article}{
   author={Connes, Alain},
   title={Classification of injective factors. Cases $II_{1},$
   $II_{\infty },$  $III_{\lambda },$  $\lambda \not=1$},
   journal={Ann. of Math. (2)},
   volume={104},
   date={1976},
   pages={73--115},
}

\bib{G83}{article}{
   author={Giordano, T.},
   title={Antiautomorphismes involutifs des facteurs de von Neumann
   injectifs. I},
   journal={J. Operator Theory},
   volume={10},
   date={1983},
   pages={251--287},
}

\bib{ksconnes}{article}{
  author={Klep, Igor},
  author={Schweighofer, Markus},
  title={Connes' embedding conjecture and sums of Hermitian squares},
  journal={Adv. Math.},
  volume={217},
  year={2008},
  pages={1816--1837},
}

\bib{Li}{book}{
   author={Li, Bingren},
   title={Real operator algebras},
   publisher={World Scientific Publishing Co. Inc.},
   place={River Edge, NJ},
   date={2003},
}

\bib{Oz}{article}{
  author={Ozawa, Narutaka},
  title={About the Connes embedding conjecture --- algebraic approaches},
  eprint={http://arxiv.org/abs/1212.1703}
}


\bib{St67}{article}{
   author={St{\o}rmer, Erling},
   title={On anti-automorphisms of von Neumann algebras},
   journal={Pacific J. Math.},
   volume={21},
   date={1967},
   pages={349--370},
}
		
\bib{St68}{article}{
   author={St{\o}rmer, Erling},
   title={Irreducible Jordan algebras of self-adjoint operators},
   journal={Trans. Amer. Math. Soc.},
   volume={130},
   date={1968},
   pages={153--166},
}
		
\bib{St80}{article}{
   author={St{\o}rmer, Erling},
   title={Real structure in the hyperfinite factor},
   journal={Duke Math. J.},
   volume={47},
   date={1980},
   pages={145--153},
}

\end{biblist}
\end{bibdiv}
\end{document}